\documentclass[leqno, a4paper, 12pt]{article}

\usepackage{amssymb, amscd, amsthm}

\theoremstyle{plain}
\newtheorem{theorem}{Theorem}[section]
\newtheorem{lemma}[theorem]{Lemma}
\newtheorem{proposition}[theorem]{Proposition}
\newtheorem{corollary}[theorem]{Corollary}

\theoremstyle{definition}

\newtheorem{remark}[theorem]{Remark}

\newcommand\bG{{\mathbb G}}

\newcommand\bP{{\mathbb P}}

\newcommand\bZ{{\mathbb Z}}

\newcommand\cO{{\mathcal O}}

\newcommand\aff{{\rm aff}}

\newcommand\ant{{\rm ant}}
\renewcommand\char{{\rm char}}

\newcommand\et{\rm{\acute{e}t}}
\newcommand\id{{\rm id}}

\newcommand\red{{\rm red}}

\newcommand\Aut{{\rm Aut}}

\newcommand\Coker{{\rm Coker}}

\newcommand\Ext{{\rm Ext}}

\newcommand\Gal{{\rm Gal}}
\newcommand\GL{{\rm GL}}
\newcommand\Hilb{{\rm Hilb}}
\newcommand\Hom{{\rm Hom}}

\newcommand\Ker{{\rm Ker}}

\newcommand\SL{{\rm SL}}
\newcommand\Spec{{\rm Spec}}

\newcommand\Sym{{\rm Sym}}

\newcommand\Tr{{\rm Tr}}

\title{On extensions of algebraic groups with finite quotient}

\author{Michel Brion}

\date{}

\begin{document}

\maketitle
 
\begin{abstract}
We obtain a lifting property for finite quotients of algebraic groups, 
and applications to the structure of these groups. 
\end{abstract}

\section{Introduction}
\label{sec:int}

Consider an extension of algebraic groups, that is, an exact sequence 
of group schemes of finite type over a field,
\begin{equation}\label{eqn:ext} \CD
1 @>>> N @>>> G @>{f}>> Q @>>> 1.
\endCD \end{equation}
Such an extension is generally not split, i.e., $f$ admits 
no section which is a morphism of group schemes.  
In this note, we obtain the existence of a splitting in a weaker 
sense, for extensions with finite quotient group:

\begin{theorem}\label{thm:main}
Let $G$ be an algebraic group over a field $k$, and $N$ a normal
subgroup of $G$ such that $G/N$ is finite. Then there exists 
a finite subgroup $F$ of $G$ such that $G = N \cdot F$. 
\end{theorem}

Here $N \cdot F$ denotes, as in \cite[VIA.5.3.3]{SGA3},
the quotient of the semi-direct product $N \rtimes F$ by $N \cap F$ 
embedded as a normal subgroup via $x \mapsto (x,x^{-1})$. 
If $G/N$ is \'etale and $k$ is perfect, then the subgroup $F$ may be 
chosen \'etale as well. But this fails over any imperfect field $k$, 
see Remark \ref{rem:etale} for details.

In the case where $G$ is smooth and $k$ is perfect, 
Theorem \ref{thm:main} was known to Borel and Serre, and they 
presented a proof over an algebraically closed field of characteristic
$0$ (see \cite[Lem.~5.11 and footnote on p.~152]{Borel-Serre}).
That result was also obtained by Platonov for smooth linear algebraic
groups over perfect fields (see \cite[Lem.~4.14]{Platonov}).
In the latter setting, an effective version of Theorem \ref{thm:main} 
has been obtained recently by Lucchini Arteche; see 
\cite[Thm.~1.1]{Lucchini-II}, and \cite[Prop.~1.1]{Lucchini}), 
\cite[p.~473]{CGR}, \cite[Lem.~5.3]{LMMR} for earlier results 
in this direction.

Returning to an extension (\ref{eqn:ext}) with an arbitrary quotient
$Q$, one may ask whether there exists a subgroup $H$ of $G$ such 
that $G = N \cdot H$ and $N \cap H$ is finite (when $Q$ is finite, 
the latter condition is equivalent to the finiteness of $H$). 
We then say that (\ref{eqn:ext}) is \emph{quasi-split}, and 
$H$ is a \emph{quasi-complement} of $N$ in $G$, 
with \emph{defect group} $N \cap H$. 

When $Q$ is smooth and $N$ is an abelian variety, every extension 
(\ref{eqn:ext}) is quasi-split (as shown by Rosenlicht; see 
\cite[Thm.~14]{Rosenlicht}, and \cite[Sec.~2]{Milne} for a modern 
proof). The same holds when $Q$ is reductive (i.e., $Q$ is smooth 
and affine, and the radical of $Q_{\bar{k}}$ is a torus), 
$N$ is arbitrary and $\char(k) = 0$, as 
we will show in Corollary \ref{cor:red}. On the other hand, the group 
$G$ of unipotent $3 \times 3$ matrices sits in a central extension 
$1 \to \bG_a \to G \to \bG_a^2 \to 1$, 
which is not quasi-split. It would be interesting to determine which 
classes of groups $N$, $Q$ yield quasi-split extensions. 
Another natural problem is to bound the defect group in terms of $N$ 
and $Q$. The proof of Theorem \ref{thm:main} yields some information 
in that direction; see Remark \ref{rem:def}, and \cite{Lucchini-II} 
for an alternative approach via nonabelian Galois cohomology. 

This article is organized as follows. In Section \ref{sec:reductions},
we begin the proof of Theorem \ref{thm:main} with a succession of 
reductions to the case where $Q = G/N$ is \'etale and $N$ is 
a smooth connected unipotent group, a torus, or an abelian variety. 
In Section \ref{sec:commutative}, we show that every class of extensions 
(\ref{eqn:ext}) is torsion in that setting (Lemma 
\ref{lem:commutative}); this quickly implies Theorem 
\ref{thm:main}. Section \ref{sec:applications} presents 
some applications of Theorem \ref{thm:main} to the structure of 
algebraic groups: we obtain analogues of classical results of Chevalley, 
Rosenlicht and Arima on smooth connected algebraic groups
(see \cite{Rosenlicht, Arima, Rosenlicht-II}) and of Mostow
on linear algebraic groups in characteristic $0$ (see \cite{Mostow}). 
Finally, we show that every homogeneous space under an algebraic group 
admits a projective equivariant compactification; this result seems to 
have been unrecorded so far. It is well-known that any such homogeneous 
space is quasi-projective (see \cite[Cor.~VI.2.6]{Raynaud}); also, 
the existence of equivariant compactifications of certain homogeneous 
spaces having no separable point at infinity has attracted recent interest 
(see e.g. \cite{Gabber, GGM}).

\section{Proof of Theorem \ref{thm:main}: some reductions}
\label{sec:reductions}

We first fix notation and conventions, which will be used 
throughout this article. We consider schemes and their morphisms
over a field $k$, and choose an algebraic closure $\bar{k}$. 
Given a scheme $X$ and an extension $K/k$ of fields, we denote by 
$X_K$ the $K$-scheme obtained from $X$ by the base change 
$\Spec(K) \to \Spec(k)$.

We use mostly \cite{SGA3}, and occasionally \cite{DG}, as references 
for group schemes. Given such a group scheme $G$, we denote by 
$e_G \in G(k)$ the neutral element, and by $G^0$ the neutral component 
of $G$, with quotient map $\pi : G \to G/G^0 = \pi_0(G)$. 
The group law of $G$ is denoted by 
$\mu : G \times G \to G$, $(x,y) \mapsto xy$. 

Throughout this section, we consider an extension (\ref{eqn:ext})
and a subgroup $F$ of $G$. Then the map 
\[ \nu : N \rtimes F \longrightarrow G, \quad (x,y) \longmapsto xy \] 
is a morphism of group schemes with kernel
$N \cap F$, embedded in $N \rtimes F$ via $x \mapsto (x,x^{-1})$.
Thus, $\nu$ factors through a morphism of group schemes
\[ \iota : N \cdot F \longrightarrow G. \]
Also, the composition $F \to G \to G/N$ factors through a morphism
of group schemes 
\[ i : F/(N \cap F) \longrightarrow G/N. \] 
By \cite[VIA.5.4]{SGA3}, $\iota$ and $i$ are closed immersions
of group schemes.

\begin{lemma}\label{lem:equiv}
The following conditions are equivalent:

\smallskip

\noindent
{\rm (i)} $\iota$ is an isomorphism.

\smallskip

\noindent
{\rm (ii)} $i$ is an isomorphism.

\smallskip

\noindent
{\rm (iii)} $\nu$ is faithfully flat.

\smallskip

\noindent
{\rm (iv)} For any scheme $S$ and any $g \in G(S)$, there exists
a faithfully flat morphism of finite presentation $f: S' \to S$
and $x \in N(S')$, $y \in F(S')$ such that $g = x y$ in $G(S')$.

\smallskip

When $G/N$ is smooth, these conditions are equivalent to:

\smallskip

\noindent
{\rm (v)} $G(\bar{k}) = N(\bar{k}) F(\bar{k})$.

\end{lemma}

\begin{proof}
Recall that $i$ factors through an isomorphism 
$F/(N \cap F) \to (N \cdot F)/N$ (see \cite[VIA.5.5.3]{SGA3}). 
Thus, we obtain a commutative diagram
\[ \CD 
N \cdot F @>{\iota}>> G \\
@V{\varphi}VV @V{f}VV \\
F/(N \cap F) @>{i}>> G/N, \\
\endCD \]
where both vertical arrows are $N$-torsors for the action of $N$ 
by right multiplication. As a consequence, this diagram is cartesian.
In particular, $i$ is an isomorphism if and only if so is $\iota$;
this yields the equivalence (i) $\Leftrightarrow$ (ii).

(i) $\Rightarrow$ (iii): Since $\nu$ is identified with the
quotient map of $N \rtimes F$ by $N \cap F$, the assertion follows
from \cite[VIA.3.2]{SGA3}. 

(iii) $\Rightarrow$ (iv): This follows by forming the cartesian square
\[ \CD
S' @>>> S \\
@VVV @V{g}VV \\
N \rtimes F @>{\nu}>> G \\
\endCD \]
and observing that $\nu$ is of finite presentation, since 
the schemes $G$, $N$ and $F$ are of finite type. 

(iv) $\Rightarrow$ (i): By our assumption applied to the identity 
map $G \to G$, there exists a scheme 
$S'$ and morphisms $x : S' \to N$, $y : S' \to F$ such that the 
morphism $\nu \circ (x \times y) : N \rtimes F \to G$ is faithfully 
flat of finite presentation. As a consequence, the morphism of 
structure sheaves $\cO_G \to \nu_* (x \times y)_* (\cO_{S'})$ is 
injective. Thus, so are $\cO_G \to \nu_*(\cO_{N \rtimes F})$, and hence  
$\cO_G \to i_* (\cO_{N \cdot F})$. Since $i$ is a closed immersion, 
it must be an isomorphism. 

When $G/N$ is smooth, $i$ is an isomorphism 
if and only if it is surjective on $\bar{k}$-rational points. Since 
$(G/N)(\bar{k}) = G(\bar{k})/N(\bar{k})$ and likewise for 
$F/(N\cap F)$, this yields the equivalence (ii) $\Leftrightarrow$ (v).
\end{proof}

We assume from now on that the quotient group $Q$ in the
extension (\ref{eqn:ext}) is \emph{finite}.

\begin{lemma}\label{lem:smooth}
If the exact sequence $1 \to H^0 \to H \to \pi_0(H) \to 1$
is quasi-split for any smooth algebraic group $H$ such that
$\dim(H) = \dim(G)$, then (\ref{eqn:ext}) is quasi-split as well.
\end{lemma}

\begin{proof}
Consider first the case where $G$ is smooth. Then $Q$ is \'etale,
and hence $N$ contains $G^0$. By our assumption, there exists a finite 
subgroup $F \subset G$ such that $G = G^0 \cdot F$. In view of Lemma
\ref{lem:equiv} (iv), it follows that $G = N \cdot F$.

If $\char(k)=0$, then the proof is completed as every algebraic 
group is smooth (see e.g. \cite[VIB.1.6.1]{SGA3}). 
So we may assume that $\char(k) = p > 0$. 
Consider the $n$-fold relative Frobenius morphism 
\[ F^n_G : G \longrightarrow G^{(p^n)} \] 
and its kernel $G_n$. Then $F^n_G$ is finite and bijective, so that 
$G_n$ is an infinitesimal normal subgroup of $G$. Moreover, the quotient 
$G/G_n$ is smooth for $n \gg 0$ (see \cite[VIIA.8.3]{SGA3}). We may
thus choose $n$ so that $G/G_n$ and $N/N_n$ are smooth. The 
composition $N \to G \to G/G_n$ factors through a closed immersion
of group schemes $N/N_n \to G/G_n$ by \cite[VIA.5.4]{SGA3} again. 
Moreover, the image of $N/N_n$ is a normal subgroup of $G/G_n$,
as follows e.g. from our smoothness assumption and 
\cite[VIB.7.3]{SGA3}. This yields an exact sequence
\[
1 \longrightarrow N/N_n \longrightarrow G/G_n \longrightarrow Q'
\longrightarrow 1,
\] 
where $Q'$ is a quotient of $Q$ and hence is finite; moreover,
$\dim(G/G_n) = \dim(G)$. By our assumption and the first 
step, there exists a finite subgroup $F'$ of $G/G_n$ such that
$G/G_n = (N/N_n) \cdot F'$. In view of \cite[VIA.5.3.1]{SGA3},
there exists a unique subgroup $F$ of $G$ containing $G_n$ such that 
$F/G_n = F'$; then $F$ is finite as well.

We check that $G = N \cdot F$ by using Lemma \ref{lem:equiv} (iv)
again. Let $S$ be a scheme, and $g \in G(S)$. Then there exists a
faithfully flat morphism of finite presentation $S' \to S$ and
$x' \in (N/N_n)(S')$, $y' \in F'(S')$ such that$F^n_G(g) = x' y'$ in
$(G/G_n)(S')$. Moreover, there exists a faithfully flat morphism 
of finite presentation $S'' \to S'$ and $x'' \in N(S'')$, 
$y'' \in F(S'')$ such that $F^n_G(x'') = x'$ and $F^n_G(y'') = y'$. 
Then $y''^{-1} x''^{-1} g \in G_n(S'')$, and hence 
$g \in N(S'') F(S'')$, since $F$ contains $G_n$.  
\end{proof}

\begin{remark}\label{rem:smooth}
With the notation of the proof of Lemma \ref{lem:smooth}, 
there is an exact sequence of quasi-complements 
\[ 
1 \longrightarrow G_n \longrightarrow F \longrightarrow F' 
\longrightarrow 1.
\] 
When $N = G^0$, so that $G_n \subset N$, we also have an exact 
sequence of defect groups
\[
1 \longrightarrow G_n \longrightarrow N \cap F 
\longrightarrow (N/G_n) \cap F' \longrightarrow 1. 
\]
\end{remark}

By Lemma \ref{lem:smooth}, it suffices to prove Theorem 
\ref{thm:main} when $G$ \emph{is smooth and} $N = G^0$, 
so that $Q = \pi_0(G)$. We may thus choose a maximal torus $T$ 
of $G$, see \cite[XIV.1.1]{SGA3}. Then the normalizer $N_G(T)$ 
and the centralizer $Z_G(T)$ are (represented by) subgroups of $G$, 
see \cite[VIB.6.2.5]{SGA3}). Moreover, $N_G(T)$ is smooth by 
\cite[XI.2.4]{SGA3}. We now gather further properties of $N_G(T)$:

\begin{lemma}\label{lem:tor}
{\rm (i)} $G = G^0 \cdot N_G(T)$.

\smallskip

\noindent
{\rm (ii)} $N_G(T)^0= Z_{G^0}(T)$.

\smallskip

\noindent
{\rm (iii)} We have an exact sequence
$1 \to W(G^0,T) \to \pi_0(N_G(T)) \to \pi_0(G) \to 1$,
where $W(G^0,T) := N_{G^0}(T)/Z_{G^0}(T) = \pi_0(N_{G^0}(T))$
denotes the Weyl group.
\end{lemma}

\begin{proof}
(i) By Lemma \ref{lem:equiv} (v), it suffices to show that 
$G(\bar{k}) = G^0(\bar{k}) N_G(T)(\bar{k})$. Let $x \in G(\bar{k})$,
then $x T x^{-1}$ is a maximal torus of $G^0(\bar{k})$, and hence
$x T x^{-1} = y T y^{-1}$ for some $y \in G^0(\bar{k})$.
Thus, $x \in y N_G(T)(\bar{k})$, which yields the assertion.

(ii) We may assume that $k$ is algebraically closed and $G$ is
connected (since $N_G(T)^0 = N_{G^0}(T)^0$). Then $Z_G(T)$ is
a Cartan subgroup of $G$, and hence equals its connected
normalizer by \cite[XII.6.6]{SGA3}.

(iii) By (i), the natural map $N_G(T)/N_{G^0}(T) \to \pi_0(G)$
is an isomorphism. Combined with (ii), this yields the 
statement.
\end{proof}

\begin{remark}\label{rem:tor}
If $N_G(T) = N_G(T)^0 \cdot F$ for some subgroup $F \subset N_G(T)$,
then $G = G^0 \cdot F$ by Lemmas \ref{lem:equiv} and \ref{lem:tor}.
Moreover, the commutative diagram of exact sequences
\[ \CD 
1 \longrightarrow & N_G(T)^0 \cap F & \longrightarrow & F & 
\longrightarrow & \pi_0(N_G(T)) & \longrightarrow 1 \\
& @VVV @V{\id}VV @VVV & & \\
1 \longrightarrow & G^0 \cap F & \longrightarrow & F & 
\longrightarrow & \pi_0(G) & \longrightarrow 1 \\
\endCD \]
together with Lemma \ref{lem:tor} yields the exact sequence
\[
1 \longrightarrow Z_{G^0}(T) \cap F \longrightarrow G^0 \cap F 
\longrightarrow W(G^0,T) \longrightarrow 1.
\] 
\end{remark}

In view of Lemma \ref{lem:equiv} (iv) and Lemma \ref{lem:tor} (i),
it suffices to prove Theorem \ref{thm:main} under the additional
assumption that $T$ \emph{is normal in} $G$. Then $T$ is central in 
$G^0$, and hence $G^0_{\bar{k}}$ is nilpotent by \cite[XII.6.7]{SGA3}. 
It follows that $G^0$ \emph{is nilpotent}, in view of 
\cite[VIB.8.3]{SGA3}. To obtain further reductions, we will use 
the following:

\begin{lemma}\label{lem:ind}
Let $N'$ be a normal subgroup of $G$ contained in $N$. 
Assume that the resulting exact sequence 
$1 \to N/N' \to G/N' \to Q \to 1$ 
is quasi-split, and that any exact sequence of algebraic groups 
$1 \to N' \to G' \to Q' \to 1$, 
where $Q'$ is finite, is quasi-split as well. Then (\ref{eqn:ext}) 
is quasi-split.
\end{lemma}

\begin{proof}
By assumption, there exists a finite subgroup $F'$ of $G/N'$
such that $G/N' = (N/N') \cdot F'$. Denote by $G'$ the subgroup 
of $G$ containing $N'$ such that $G'/N' = F'$. By assumption 
again, there is a finite subgroup $F$ of $G'$ containing $N'$ 
such that $G' = N' \cdot F$. We check that $G = N \cdot F$ by using
Lemma \ref{lem:equiv} (iv). Let $S$ be a scheme, and $g \in G(S)$; 
denote by $f' : G \to G/N'$ the quotient map. Then there exists a 
faithfully flat morphism of finite presentation $S' \to S$ and 
$x \in (N/N')(S')$, $y \in F'(S')$ such that $f'(g) = x y$
in $(G/N')(S')$. Moreover, there exists a 
faithfully flat morphism of finite presentation $S'' \to S'$
and $z \in N(S'')$, $w \in G'(S'')$ such that $f'(z) = x$ 
and $f'(w) = y$. Then $w^{-1} z^{-1} g \in N'(S'')$, and hence
$g \in N(S'') G'(S'')$, as $G'$ contains $N'$. This shows that
$G = N \cdot G' = N \cdot (N' \cdot F)$. We conclude by
observing that $N \cdot (N' \cdot F) = N \cdot F$, in view of
Lemma \ref{lem:equiv} (iv) again.
\end{proof}

\begin{remark}\label{rem:ind}
With the notation of the proof of Lemma \ref{lem:ind}, we have
an exact sequence 
$1 \to N' \to G' = N' \cdot F \to F' \to 1$,
and hence an exact sequence of quasi-complements
\[
1 \longrightarrow N'\cap F \longrightarrow F 
\longrightarrow F' \longrightarrow 1.
\]
Moreover, we obtain an exact sequence
$1 \to N' \to N \cap G' \to (N/N') \cap F' \to 1$,
by using \cite[VIA.5.3.1]{SGA3}. Since 
$N \cap G' = N \cap (N' \cdot F) = N' \cdot (N \cap F)$,
where the latter equality follows from Lemma \ref{lem:equiv} (iv),
this yields an exact sequence of defect groups
\[
1 \longrightarrow N' \cap F \longrightarrow N \cap F 
\longrightarrow (N/N') \cap F' \longrightarrow 1.
\]
\end{remark}

Next, we show that it suffices to prove Theorem \ref{thm:main}
when $G^0$ is assumed in addition to be \emph{commutative}.
 
We argue by induction on the dimension of $G$ (assumed to be
smooth, with $G^0$ nilpotent). If $\dim(G) = 1$, then $G^0$ is 
either a $k$-form of $\bG_a$ or $\bG_m$, or an elliptic curve;
in particular, $G^0$ is commutative. In higher dimensions,  
the derived subgroup $D(G^0)$ is a smooth, connected normal 
subgroup of $G$ contained in $G^0$, and the quotient $G^0/D(G^0)$ 
is commutative of positive dimension (see \cite[VIB.7.8, 8.3]{SGA3}). 
Moreover, $G/D(G^0)$ is smooth, and $\pi_0(G/D(G^0)) = \pi_0(G)$. 
By the induction assumption, it follows that the exact sequence 
$1 \to G^0/D(G^0) \to G/D(G^0) \to \pi_0(G) \to 1$ is quasi-split.
Also, every exact sequence $1 \to D(G^0) \to G' \to Q' \to 1$,
where $Q'$ is finite, is quasi-split, by the induction assumption 
again together with Lemma \ref{lem:smooth}. Thus, Lemma 
\ref{lem:ind} yields the desired reduction.
 
We now show that we may further assume $G^0$ to be
\emph{a torus, a smooth connected commutative unipotent group, 
or an abelian variety}.

Indeed, we have an exact sequence of commutative algebraic groups
\[
1 \longrightarrow T \longrightarrow G^0 \longrightarrow H
\longrightarrow 1,  
\]
where $T$ is the maximal torus of $G^0$, and $H$ is smooth and
connected. Moreover, we have an exact sequence
\[
1 \longrightarrow H_1 \longrightarrow H \longrightarrow H_2
\longrightarrow 1,  
\]
where $H_1$ is a smooth connected affine algebraic group, and $H_2$ 
is a pseudo-abelian variety in the sense of \cite{Totaro}, i.e.,
$H_2$ has no nontrivial smooth connected affine normal subgroup. 
Since $H_1$ contains no nontrivial torus, it is unipotent;
also, $H_2$ is an extension of a smooth connected unipotent
group by an abelian variety $A$, in view of \cite[Thm.~2.1]{Totaro}.
Note that $T$ is a normal subgroup of $G$ (the largest subtorus). 
Also, $H_1$ is a normal subgroup of $G/T$ (the largest smooth 
connected affine normal subgroup of the neutral component), 
and $A$ is a normal subgroup of $(G/T)/H_1$ as well (the largest 
abelian subvariety). 
Thus, arguing by induction on the dimension as in the preceding step, 
with $D(G^0)$ replaced successively by $T$, $H_1$ and $A$, yields 
our reduction. 

When $G^0$ is unipotent and $\char(k) = p > 0$, we may further 
assume that $G^0$ \emph{is killed by} $p$. Indeed, by 
\cite[XVII.3.9]{SGA3}, there exists a composition series
$\{ e_G \} = G_0 \subset G_1 \subset \cdots \subset G_n = G^0$
such that each $G_i$ is normal in $G$, and each quotient
$G_i/G_{i-1}$ is a $k$-form of some $(\bG_a)^{r_i}$; in 
particular, $G_i/G_{i-1}$ is killed by $p$. Our final reduction 
follows by induction on $n$.

\section{Proof of Theorem \ref{thm:main}: 
extensions by commutative groups}
\label{sec:commutative}

In this section, we consider smooth algebraic groups $Q$, $N$ 
such that $Q$ is finite and $N$ is commutative.
Given an extension (\ref{eqn:ext}), 
the action of $G$ on $N$ by conjugation factors through an action 
of $Q$ by group automorphisms that we denote by $(x,y) \mapsto y^x$,
where $x \in Q$ and $y \in N$. 
Recall that the isomorphism classes 
of such extensions with a prescribed $Q$-action on $N$
form a commutative (abstract) group, that we denote by 
$\Ext^1(Q,N)$; see \cite[XVII.App.~I]{SGA3} (and 
\cite[III.6.1]{DG} for the setting of extensions of group sheaves).

\begin{lemma}\label{lem:commutative}
With the above notation and assumptions, the group
$\Ext^1(Q,N)$ is torsion. 
\end{lemma}

\begin{proof}
Any extension (\ref{eqn:ext}) yields an $N$-torsor over $Q$
for the \'etale topology, since $Q$ is finite and \'etale.
This defines a map $\tau : \Ext^1(Q,N) \to H^1_{\et}(Q,N)$,
which is a group homomorphism (indeed, the sum of any two 
extensions is obtained by taking their direct product, 
pulling back under the diagonal map $Q \to Q \times Q$, and
pushing forward under the multiplication $N \times N \to N$,
and the sum of any two torsors is obtained by the analogous
operations). The kernel of $\tau$ consists of those classes 
of extensions that admit a section (which is a morphism of 
schemes). In view of \cite[XVII.App.~I.3.1]{SGA3}, 
this yields an exact sequence
\[ 0 \longrightarrow HH^2(Q,N) \longrightarrow \Ext^1(Q,N)
\stackrel{\tau}{\longrightarrow} H^1_{\et}(Q,N), \]
where $HH^i$ stands for Hochschild cohomology (denoted by $H^i$
in \cite{SGA3}, and by $H^i_0$ in \cite{DG}). Moreover, the group 
$H^1_{\et}(Q,N)$ is torsion (as follows e.g. from 
\cite[Thm.~14]{Rosenlicht}), and $HH^2(Q,N)$ is killed by the 
order of $Q$, as a special case of \cite[XVII.5.2.4]{SGA3}.
\end{proof}

\begin{remark}\label{rem:tors}
The above argument yields that $\Ext^1(Q,N)$ is killed by $md$ if
$Q$ is finite \'etale of order $m$, and $N$ is a torus split by an
extension of $k$ of degree $d$. Indeed, we just saw that 
$HH^2(Q,N)$ is killed by $m$; also, $H^1_{\et}(Q,N)$ is a direct sum 
of groups of the form $H^1_{\et}(\Spec(k'),N)$ for finite separable 
extensions $k'$ of $k$, and these groups are killed by $d$. 
This yields a slight generalization of \cite[Prop.~1.1]{Lucchini}, 
via a different approach.   
\end{remark}

\medskip

\noindent
{\sc End of the proof of Theorem \ref{thm:main}}.

Recall from our reductions in Section \ref{sec:reductions} 
that we may assume $G^0$ to be a smooth commutative unipotent group, 
a torus, or an abelian variety. We will rather denote $G^0$ by $N$, 
and $\pi_0(G)$ by $Q$.

We first assume in addition that $\char(k) = 0$ if $N$ is unipotent.
Then the $n$th power map 
\[ n_N: N \longrightarrow N, \quad x \longmapsto x^n \] 
is an isogeny for any positive integer $n$. Consider an extension 
(\ref{eqn:ext}) and denote by $\gamma$ its class in $\Ext^1(Q,N)$. 
By Lemma \ref{lem:commutative}, we may choose $n$ so that 
$n \gamma = 0$. Also, $n \gamma = (n_N)_*(\gamma)$ (the pushout 
of $\gamma$ by $n_N$); moreover, the exact sequence of commutative 
algebraic groups
\[ \CD 
1 @>>> N[n] @>>> N @>{n_N}>> N @>>> 1 
\endCD \]
yields an exact sequence 
\[ \CD 
\Ext^1(Q,N[n]) @>>> \Ext^1(Q,N) @>{(n_N)_*}>> \Ext^1(Q,N)
\endCD \]
in view of \cite[XVII.App.~I.2.1]{SGA3}. Thus, there exists a class 
$\gamma' \in \Ext^1(Q,N[n])$ with pushout $\gamma$, i.e., we have 
a commutative diagram of extensions
\[ \CD
1 \longrightarrow & N[n] & \longrightarrow & G' &
\longrightarrow & Q & \longrightarrow 1 \\
& @VVV    @VVV   @V{\id}VV & & \\  
1 \longrightarrow & N  & \longrightarrow  & G & \longrightarrow & Q &
\longrightarrow 1, \\
\endCD \]
where the square on the left is cartesian. It follows that $G'$ 
is a finite subgroup of $G$, and $G = N \cdot G'$.

Next, we consider the remaining case, where $N$ is unipotent and
$\char(k) = p > 0$. In view of our final reduction at the end of 
Section 2, we may further assume that $N$ is killed by $p$. 
Then there exists an \'etale isogeny $N \to N_1$, where $N_1$ 
is a vector group (see \cite[Lem.~B.1.10]{CGP}). This yields 
another commutative diagram of extensions
\[ \CD
1 \longrightarrow & N & \longrightarrow & G &
\longrightarrow & Q & \longrightarrow 1 \\
& @VVV    @VVV   @V{\id}VV & & \\  
1 \longrightarrow & N_1  & \longrightarrow  & G_1 & \longrightarrow & Q &
\longrightarrow 1. \\
\endCD \]
Assume that there exists a finite subgroup $F_1$ of $G_1$ such that
$G_1 = N_1 \cdot F_1$. Let $F$ be the pullback of $F_1$ to $G$; then
$F$ is a finite subgroup, and one checks that $G = N \cdot F$ by using 
Lemma \ref{lem:equiv} (iv). Thus, we may finally assume that 
$N$ \emph{is a vector group}.

Under that assumption, the $N$-torsor $G \to Q$ is trivial, since
$Q$ is affine. Thus, we may choose a section $s : Q \to G$.
Also, we may choose a finite Galois extension of fields $K/k$ such
that $Q_K$ is constant. Then $s$ yields a section 
$s_K : Q_K \to G_K$, equivariant under the Galois group 
$\Gamma_K : = \Gal(K/k)$. So we may view $G(K)$ as the set of the 
$y \, s(x)$, where $y \in N(K)$ and $x \in Q(K)$, with multiplication
\[ y \, s(x) \, y' \, s(x') = y \, y'^x \, c(x,x') \, s(x x'), \]
where $c \in Z^2(Q(K),N(K))^{\Gamma_K}$. Consider the (abstract)
subgroup $H \subset N(K)$ generated by the $c(x', x'')^x$,
where $x,x',x'' \in Q(K)$. Then $H$ is finite, since $N(K)$
is killed by $p$ and $Q(K)$ is finite. Moreover, $H \, s(Q(K))$
is a subgroup of $G(K)$, in view of the above formula for the
multiplication. Clearly, $H \, s(Q(K))$ is finite and stable 
under $\Gamma_K$; thus, it corresponds to a finite (algebraic)
subgroup $G'$  of $G$. Also, we obtain as above that 
$G = N \cdot G'$. This completes the proof of Theorem 
\ref{thm:main}.

\begin{remark}\label{rem:etale}
If $k$ is perfect, then the subgroup $F$ as in Theorem 
\ref{thm:main} may be chosen \'etale. Indeed, the reduced subscheme 
$F_{\red}$ is then a subgroup by \cite[VIA.0.2]{SGA3}. Moreover,
$G(\bar{k}) = G^0(\bar{k}) F_{\red}(\bar{k})$, and hence 
$G = G^0 \cdot F_{\red}$ in view of Lemma \ref{lem:equiv} (v).

In contrast, when $k$ is imperfect, there exists a finite group $G$ 
admitting no \'etale subgroup $F$ such that $G = G^0 \cdot F$. 
Consider for example (as in \cite[VIA.1.3.2]{SGA3}) the subgroup 
$G$ of $\bG_{a,k}$ defined by the additive polynomial 
$X^{p^2} - t X^p$, where $p := \char(k)$ and $t \in k \setminus k^p$.
Then $G$ has order $p^2$ and $G^0$ has order $p$. 
If $G = G^0 \cdot F$ with $F$ \'etale, then $G^0 \cap F$ is trivial. 
Thus, $G \cong G^0 \rtimes F$ and $F$ has order $p$. Let 
$K := k(t^{1/p})$, then $F_K$ is contained in $(G_K)_{\red}$, which is 
the subgroup of $\bG_{a,K}$ defined by the additive polynomial 
$X^p - t^{1/p} X$. By counting dimensions, it follows that 
$F_K = (G_K)_{\red}$, which yields a contradiction as $(G_K)_{\red}$ 
is not defined over $k$.
\end{remark}

\begin{remark}\label{rem:def}
One may obtain information on the defect group $N \cap F$ 
by examining the steps in the proof of Theorem \ref{thm:main}
and combining Remarks \ref{rem:smooth}, \ref{rem:tor} and \ref{rem:ind}.
For instance, if $G$ is smooth, then $N \cap F$ is an extension of
the Weyl group $W(G^0,T)$ by the nilpotent group $Z_{G^0}(T) \cap F$,
where $T$ is a maximal torus of $G$. If $\char(k) = 0$ (so that $G$
is smooth), then $Z_{G^0}(T) \cap F$ is commutative. Indeed, $Z_{G^0}(T)$
is a connected nilpotent algebraic group, and hence an extension of
a semi-abelian variety $S$ by a connected unipotent algebraic group
$U$. Thus, $U \cap F$ is trivial, and hence $Z_{G^0}(T) \cap F$ is 
isomorphic to a subgroup of $S$. 
\end{remark}

\begin{remark}\label{rem:finite}
When $k$ is finite, Theorem \ref{thm:main} follows readily 
from our first reduction step (Lemma \ref{lem:smooth}) together 
with a theorem of Lang, see \cite[Thm.~2]{Lang}. More specifically, 
let $H$ be a smooth algebraic group and choose representatives
$x_1,\ldots,x_m$ of the orbits of the Galois group 
$\Gamma := \Gal(\bar{k}/k)$ in 
$\pi_0(H)(\bar{k})$. Denote by $\Gamma_i \subset \Gamma$ the isotropy 
group of $x_i$ and set $k_i := \bar{k}^{\Gamma_i}$ for $i = 1, \ldots, m$. 
Then $x_i \in \pi_0(H)(k_i)$, and hence the fiber $\pi_{k_i}^{-1}(x_i)$ 
(a torsor under $H^0_{k_i}$) contains a $k_i$-rational point. Consider 
the subfield $K := \prod_{i=1}^n k_i \subset \bar{k}$. Then the finite 
\'etale group scheme $\pi_0(H)_K$ is constant, and $\pi$ is surjective
on $K$-rational points. Thus, $\pi_0(H)$ has a quasi-complement in $H$:
the finite \'etale group scheme corresponding to the constant,
$\Gamma$-stable subgroup scheme $H(K)$ of $H_K$.    
\end{remark}

\section{Some applications}
\label{sec:applications}

We first recall two classical results on the structure of algebraic groups.
The first one is the affinization theorem  (see \cite[III.3.8]{DG} and also
\cite[VIB.12.2]{SGA3}): any algebraic group $G$ has a smallest normal
subgroup $H$ such that $G/H$ is affine. Moreover, $H$ is smooth, 
connected and contained in the center of $G^0$; we have $\cO(H) = k$
(such an algebraic group is called \emph{anti-affine}) and 
$\cO(G/H) = \cO(G)$. 

As a consequence, $H$ is the fiber at $e_G$ of the affinization 
morphism $G \to \Spec \, \cO(G)$; moreover, the formation of $H$ 
commutes with arbitrary field extensions. Also, note that $H$ is 
the largest anti-affine subgroup of $G$; we will denote $H$ 
by $G_{\ant}$. The structure of anti-affine groups is described 
in \cite{Brion} and \cite{Sancho}.
 
The second structure result is a version of a theorem of Chevalley, 
due to Raynaud (see  \cite[Lem.~IX.2.7]{Raynaud} and also 
\cite[9.2 Thm.~1]{BLR}): any connected algebraic group $G$ has 
a smallest affine normal subgroup $N$ such that $G/N$ is an abelian
variety. Moreover, $N$ is connected; if $G$ is smooth and $k$ is perfect,
then $N$ is smooth as well. We will denote $N$ by $G_{\aff}$.

We will also need the following observation:

\begin{lemma}\label{lem:quotient}
Let $G$ be an algebraic group, and $N$ a normal subgroup. Then
the quotient map $f: G \to G/N$ yields an isomorphism
$G_{\ant}/(G_{\ant} \cap N) \cong (G/N)_{\ant}$.
\end{lemma}

\begin{proof}
We have a closed immersion of group schemes 
$G_{\ant}/(G_{\ant} \cap N) \to G/N$; moreover, $G_{\ant}/(G_{\ant} \cap N)$
is anti-affine. So we obtain a closed immersion of commutative
group schemes $i : G_{\ant}/(G_{\ant} \cap N) \to (G/N)_{\ant}$. 
The cokernel of $i$ is anti-affine, as a quotient of $(G/N)_{\ant}$. 
Also, this cokernel is a subgroup of $(G/N)/(G_{\ant}/(G_{\ant} \cap N))$,
which is a quotient of $G/G_{\ant}$. Since the latter is affine, it
follows that $\Coker(i)$ is affine as well, by using 
\cite[VIB.11.17]{SGA3}. Thus, $\Coker(i)$ is trivial, i.e., 
$i$ is an isomorphism.
\end{proof}

We now obtain a further version of Chevalley's structure theorem, 
for possibly nonconnected algebraic groups:

\begin{theorem}\label{thm:chevalley}
Any algebraic group $G$ has a smallest affine normal subgroup $N$ 
such that $G/N$ is proper. Moreover, $N$ is connected.
\end{theorem}

\begin{proof}
It suffices to show that $G$ admits an affine normal subgroup
$N$ such that $G/N$ is proper. Indeed, given another such subgroup
$N'$, the natural map 
\[ G/(N \cap N') \longrightarrow G/N \times G/N' \]
is a closed immersion, and hence $G/(N \cap N')$ is 
proper. Taking for $N$ a minimal such subgroup, it follows
that $N$ is the smallest one. Moreover, the natural morphism 
$G/N^0 \to G/N$ is finite, since it is a torsor under the finite
group $N/N^0$ (see \cite[VIA.5.3.2]{SGA3}). As a consequence, 
$G/N^0$ is proper; hence $N  = N^0$ by the minimality assumption. 
Thus, $N$ is connected. 

Also, we may reduce to the case where $G$ is smooth by using 
the relative Frobenius morphism as in the proof of Lemma
\ref{lem:smooth}. 

If in addition $G$ is connected, then we just take $N = G_{\aff}$.
In the general case, we consider the (possibly non-normal)
subgroup $H := (G^0)_{\aff}$; then the homogeneous space $G/H$ 
is proper, since $G/G^0$ is finite and $G^0/H$ is proper.  
As a consequence, the automorphism functor of 
$G/H$ is represented by a group scheme $\Aut_{G/H}$, locally of 
finite type; in particular, the neutral component $\Aut^0_{G/H}$
is an algebraic group (see \cite[Thm.~3.7]{MO}). The action of $G$ 
by left multiplication on $G/H$ yields a morphism of group schemes
\[ \varphi : G \longrightarrow \Aut_{G/H}. \]
The kernel $N$ of $\varphi$ is a closed subscheme of $H$,
and hence is affine. To complete the proof, it suffices to show
that $G/N$ is proper. In turn, it suffices to check that $(G/N)^0$ 
is proper. Since $(G/N)^0 \cong G^0/(G^0 \cap N)$, and
$G^0 \cap N$ is the kernel of the restriction
$G^0 \to \Aut^0_{G/H}$, we are reduced to showing that
$\Aut^0_{G/H}$ is proper (by using \cite[VIA.5.4.1]{SGA3} again).

We claim that $\Aut^0_{G/H}$ is an abelian variety. Indeed,  
$(G/H)_{\bar{k}}$ is a finite disjoint union of copies 
of $(G^0/H)_{\bar{k}}$, which is an abelian variety. Also,
the natural morphism $A \to \Aut^0_A$ is an isomorphism
for any abelian variety $A$. Thus, $(\Aut^0_{G/H})_{\bar{k}}$ 
is an abelian variety (a product of copies of $(G^0/H)_{\bar{k}}$); 
this yields our claim, and completes the proof.      
\end{proof}

\begin{remark}\label{rem:chevalley}
The formation of $G_{\aff}$ (for a connected group scheme $G$)
commutes with separable algebraic field extensions, as follows
from a standard argument of Galois descent. But this formation
does not commute with purely inseparable field extensions, 
in view of \cite[XVII.C.5]{SGA3}.

Likewise, the formation of $N$ as in Theorem \ref{thm:chevalley}
commutes with separable algebraic field extensions. As a
consequence, $N = (G^0)_{\aff}$ for any smooth group scheme $G$
(since $(G^0)_{\aff}$ is invariant under any automorphism of $G$, 
and hence is a normal subgroup of $G$ when $k$ is separably closed).
In particular, if $k$ is perfect and $G$ is smooth, then $N$ is 
smooth as well.

For an arbitrary group scheme $G$, we may have $N \neq (G^0)_{\aff}$,
e.g. when $G$ is infinitesimal: then $N$ is trivial, while
$(G^0)_{\aff} = G$.

We do not know if the formations of $G_{\aff}$ and $N$ commute with
arbitrary separable field extensions.
\end{remark}

The structure of proper algebraic groups is easily described as follows:

\begin{proposition}\label{prop:proper}
Given a proper algebraic group $G$, there exists an abelian 
variety $A$, a finite group $F$ equipped with an action
$F \to \Aut_A$ and a normal subgroup $D \subset F$ such that 
$D$ acts faithfully on $A$ by translations and
$G \cong (A \rtimes F)/D$, where $D$ is embedded in 
$A \rtimes F$ via $x \mapsto (x,x^{-1})$. Moreover, 
$A = G_{\ant}$ and $F/D \cong G/G_{\ant}$ are uniquely determined 
by $G$. Finally, $G$ is smooth if and only if $F/D$ is \'etale.
\end{proposition}

\begin{proof}
Note that $G_{\ant}$ is a smooth connected proper algebraic group,
and hence an abelian variety. Moreover, the quotient group $G/G_{\ant}$
is affine and proper, hence finite. By Theorem \ref{thm:main}, there exists
a finite subgroup $F \subset G$ such that $G = G_{\ant} \cdot F$.
In particular, $G \cong (F \ltimes G_{\ant})/(F \cap G_{\ant})$; this implies 
the existence assertion. For the uniqueness, just note that 
$\cO(G) \cong \cO(G/A) \cong \cO(F/D)$, and this identifies the
affinization morphism to the natural homomorphism $G \to F/D$,
with kernel $A$.

If $G$ is smooth, then so is $G/A \cong F/D$; as $F/D$ is finite,
it must be \'etale. Since the homomorphism $G \to F/D$ is smooth,
the converse holds as well.  
\end{proof}

\begin{remark}\label{rem:proper}
The simplest examples of proper algebraic groups are the semi-direct
products $G = A \rtimes F$, where $F$ is a finite group acting on the
abelian variety $A$. If this action is non-trivial (for example, if $A$ is 
non-trivial and $F$ is the constant group $\bZ/2\bZ$ acting via 
$x \mapsto x^{\pm 1}$), then every morphism of algebraic groups
$f : G \to H$, where $H$ is connected, has a nontrivial kernel.
(Otherwise, $A$ is contained in the center of $G$ by the affinization
theorem). This yields examples of algebraic groups 
which admit no faithful representation in a connected algebraic group.
\end{remark}

\begin{remark}\label{rem:hom}
With the notation and assumptions of Proposition \ref{prop:proper}, 
consider a subgroup $H \subset G$ and the homogeneous space 
$X := G/H$. Then there exists an abelian variety $B$ quotient of $A$, 
a subgroup $I \subset F$ containing $D$, and a faithful homomorphism 
$I \to \Aut_B$ such that the scheme $X$ is isomorphic to the associated 
fiber bundle $F \times^I B$. Moreover, the schemes $F/I$ and $B$ are 
uniquely determined by $X$, and $X$ is smooth if and only if $F/I$ is 
\'etale.

Indeed, let $K := A \cdot H$, then 
$X \cong G \times^K K/H \cong F \times^I K/H$, where 
$I := F \cap K$. Moreover, $K/H \cong A/(A \cap H)$ is an abelian
variety. This shows the existence assertion; those on uniqueness
and smoothness are checked as in the proof of Proposition 
\ref{prop:proper}.

Conversely, given a finite group $F$ and a subgroup $I \subset F$
acting on an abelian variety $B$, the associated fiber bundle 
$F \times^I B$ exists (since it is the quotient of the projective 
scheme $F \times B$ by the finite group $I$), and is homogeneous 
whenever $F/I$ is \'etale (since $(F \times^I B)_{\bar{k}}$ is just a
disjoint union of copies of $B_{\bar{k}}$). We do not know how to
characterize the homogeneity of $F\times^ I B$ when the quotient
$F/I$ is arbitrary.
\end{remark}

Returning to an arbitrary algebraic group $G$, we have the
``Rosenlicht decomposition'' $G = G_{\ant} \cdot G_{\aff}$ 
when $G$ is connected (see e.g. \cite{Brion}). We now extend
this result to possibly nonconnected groups:

\begin{theorem}\label{thm:decomposition}
Let $G$ be an algebraic group. Then there exists an affine subgroup 
$H$ of $G$ such that $G = G_{\ant} \cdot H$. If $G$ is smooth 
and $k$ is perfect, then $H$ may be chosen smooth. 
\end{theorem}

\begin{proof}
By Theorem \ref{thm:chevalley}, we may choose an affine normal 
subgroup $N \subset G$ such that $G/N$ is proper. 
In view of Proposition \ref{prop:proper}, there exists a finite subgroup 
$F$ of $G/N$ such that $G/N = (G/N)_{\ant} \cdot F$, and
$(G/N)_{\ant}$ is an abelian variety. Let $H$ be the subgroup of $G$ 
containing $N$ such that $G/H = F$. Then $H$ is affine, 
since it sits in an extension $1 \to N \to H \to F \to 1$. 
We check that $G = G_{\ant} \cdot H$ by using Lemma 
\ref{lem:equiv} (iv).
Let $S$ be a scheme, and $g \in G(S)$. Denote by $g'$ 
the image of $g$ in $(G/N)(S)$. Then there exist
a faithfully flat morphism of finite presentation $S' \to S$ and 
$x' \in (G/N)_{\ant}(S')$, $y' \in F(S')$ such that 
$g' = x' y'$ in $(G/N)(S')$. Moreover, in view of Lemma
\ref{lem:quotient}, $x'$ lifts to some $x'' \in G_{\ant}(S'')$, 
where $S'' \to S'$ is faithfully flat of finite presentation.
So $g x"^{-1} \in G(S'')$ lifts $y'$, and hence 
$g \in G_{\ant}(S'') H(S'')$. 

If $G$ is smooth and $k$ is perfect, then $N$ may be chosen 
smooth by Theorem \ref{thm:chevalley} again; also, $F$ may be
chosen smooth by Remark \ref{rem:etale}. Then $H$ is smooth as well.
\end{proof}

We now derive from Theorem \ref{thm:decomposition} a generalization 
of our main theorem \ref{thm:main}, under the additional assumption 
of characteristic $0$ (then reductivity is equivalent to linear
reductivity, also known as full reducibility):

\begin{corollary}\label{cor:red}
Every extension (\ref{eqn:ext}) with reductive quotient group $Q$
is quasi-split when $\char(k) = 0$. 
\end{corollary}

\begin{proof}
Choose an affine subgroup $H \subset G$ such that 
$G = G_{\ant} \cdot H$ and denote by $R_u(H)$ its unipotent radical.
By a result of Mostow (see \cite[Thm.~6.1]{Mostow}), $H$ has a Levi 
subgroup, i.e., a fully reducible algebraic subgroup $L$ such that 
$H = R_u(H) \rtimes L$. Note that $R_u(H)$ is normal in $G$, since
it is normalized by $H$ and centralized by $G_{\ant}$. It follows that
$G_{\ant} \cdot R_u(H)$ is normal in $G$, and 
$G = (G_{\ant} \cdot R_u(H)) \cdot L$. Also, note that the quotient
map $f : G \to Q$ sends $G_{\ant}$ to $e_Q$ (since $Q$ is affine),
and $R_u(H)$ to $e_Q$ as well (since $Q$ is reductive). It follows
that the sequence 
\[ 1 \longrightarrow N \cap L \longrightarrow L 
\stackrel{f}{\longrightarrow} Q \to 1 \]
is exact, where $N = \Ker(f)$. If $N \cap L$ has a quasi-complement
$H$ in $L$, then $H$ is a quasi-complement to $N$ in $G$
(as follows e.g. from Lemma \ref{lem:equiv} (v)). Thus, we may assume
that $G$ is \emph{reductive}. Since every quasi-complement to $N^0$
in $G$ is a quasi-complement to $N$, we may further assume that $N$
is \emph{connected}.

We have a canonical decomposition 
\[ G^0 = D(G^0) \cdot R(G^0), \]
where the derived subgroup $D(G^0)$ is semisimple, the 
radical $R(G^0)$ is a central torus, and $D(G^0) \cap R(G^0)$ 
is finite (see e.g. \cite[XXII.6.2.4]{SGA3}). 
Thus, $G = D(G^0) \cdot (R(G^0) \cdot F)$, where 
$F \subset G$ is a quasi-complement to $G^0$. Likewise, 
$N = D(N) \cdot R(N)$, where $D(N) \subset D(G^0)$, 
$R(N) \subset R(G^0)$ and both are normal in $G$. 
Denote by $S$ the neutral component of the centralizer of 
$D(N)$ in $D(G^0)$. Then $S$ is a normal semisimple subgroup of 
$G$, and a quasi-complement to $D(N)$ in $D(G^0)$. 
If $R(N)$ admits a quasi-complement $T$ in $R(G^0) \cdot F$, 
then one readily checks that $S \cdot T$ is a quasi-complement 
to $N$ in $G$. As a consequence, we may assume replace $G$ with 
$R(G^0) \cdot F$, and hence assume that $G^0$ is a \emph{torus}.

Denote by $X^*(G^0)$ the character group of $G^0_{\bar{k}}$;
this is a free abelian group of finite rank, equipped with
a continuous action of $F(\bar{k}) \rtimes \Gamma$, 
where $\Gamma$ denotes the absolute Galois group of $k$. 
Moreover, we have a surjective homomorphism 
$\rho : X^*(G^0) \to X^*(N)$, equivariant for 
$F(\bar{k}) \rtimes \Gamma$. Thus, $\rho$ splits
over the rationals, and hence there exists a subgroup
$\Lambda \subset X^*(G^0)$, stable by 
$F(\bar{k}) \rtimes \Gamma$, which is mapped
isomorphically by $\rho$ to a subgroup of finite index
of $X^*(N)$. The quotient $X^*(G^0)/\Lambda$ corresponds
to a subtorus $H \subset G^0$, normalized by $G$, which 
is a quasi-complement to $N$ in $G^0$. So $H \cdot F$ 
is the desired quasi-complement to $N$ in $G$. 
\end{proof}

\begin{remark}\label{rem:red}
Corollary \ref{cor:red} does not extend to positive 
characteristics, due to the existence of groups
without Levi subgroups (see \cite[App.~A.6]{CGP}, 
\cite[Sec.~3.2]{McNinch}). As a specific example, 
when $k$ is perfect of characteristic $p >0$, 
there exists a non-split extension of algebraic groups
\[ 1 \longrightarrow V \longrightarrow  G 
\stackrel{f}{\longrightarrow} \SL_2 \longrightarrow 1, \]
where $V$ is a vector group on which $\SL_2$ acts linearly
via the Frobenius twist of its adjoint representation.
We show that this extension is not quasi-split: otherwise,
let $H$ be a quasi-complement to $N$ in $G$. Then so
is the reduced neutral component of $H$, and hence we
may assume that $H$ is smooth and connected. The quotient 
map $f$ restricts to an isogeny $H \to \SL_2$, and hence to 
an isomorphism. Thus, the above extension is split, 
a contradiction.
\end{remark}

Next, we obtain an analogue of the Levi decomposition (see 
\cite{Mostow} again) for possibly nonlinear algebraic groups:
 
\begin{corollary}\label{cor:mostow}
Let $G$ be an algebraic group over a field of characteristic
$0$. Then $G = R \cdot S$, where $R \subset G$ is the largest 
connected solvable normal subgroup, and $S \subset G$ is
an algebraic subgroup such that $S^0$ is semisimple; also,
$R \cap S$ is finite. 
\end{corollary}

\begin{proof}
By a standard argument, $G$ has a largest connected 
solvable normal subgroup $R$. The quotient $G/R$ is
affine, since $R \supset G_{\ant}$. Moreover, $R/G_{\ant}$ 
contains the radical of $G/G_{\ant}$, and hence 
$(G/R)^0$ is semisimple. In particular, $G/R$ is reductive.
So Corollary \ref{cor:red} yields the existence of the
quasi-complement $S$. 
\end{proof}

\begin{remark}\label{rem:mostow}
One may ask for a version of Corollary \ref{cor:mostow} 
in which the normal subgroup $R$ is replaced with an analogue 
of the unipotent radical of a linear algebraic group, and 
the quasi-complement $S$ is assumed to be reductive. But 
such a version would make little sense when $G$ is an
anti-affine semi-abelian variety (for example, when
$G$ is the extension of an abelian variety $A$ by $\bG_m$, 
associated with an algebraically trivial line bundle 
of infinite order on $A$). Indeed, such a group $G$ has a 
largest connected reductive subgroup: its maximal torus, 
which admits no quasi-complement.  

Also, recall that the radical $R$ may admit no complement in 
$G$, e.g. when $G = \GL_n$ with $n \geq 2$. 

Finally, one may also ask for the uniqueness of a minimal
quasi-complement in Corollary \ref{cor:mostow}, up to conjugacy
in $R(k)$ (as for Levi complements, see \cite[Thm.~6.2]{Mostow}).
But this fails when $k$ is algebraically closed 
and $G$ is the semi-direct product of an abelian variety $A$ 
with a group $F$ of order $2$. Denote by $\sigma$ the involution
of $A$ induced by the non-trivial element of $F$; then $R = A$,
and the complements to $R$ in $G$ are exactly the subgroups 
generated by the involutions $x \sigma$ where $x \in A^{-\sigma}(k)$,
i.e., $\sigma(x) = x^{-1}$. The action of $R(k)$ on complements 
is given by $y x \sigma y^{-1} = x y \sigma(y)^{-1} \sigma$;
moreover, the homomorphism $A \to A^{-\sigma}$, 
$y \mapsto y \sigma(y)^{-1}$ is generally not surjective.
This holds for example when $A = (B \times B)/C$, where $B$ is 
a nontrivial abelian variety, $C$ is the subgroup of $B \times B$ 
generated by $(x_0,x_0)$ for some $x_0 \in B(k)$ of order $2$,
and $\sigma$ arises from the involution 
$(x,y) \mapsto (y^{-1},x^{-1})$ of $B \times B$; then 
$A^{-\sigma}$ has $2$ connected components.
\end{remark}

Another consequence of Theorem \ref{thm:decomposition} concerns 
the case where $k$ is finite; then every anti-affine 
algebraic group is an abelian variety (see \cite[Prop.~2.2]{Brion}). 
This yields readily:

\begin{corollary}\label{cor:finite}
Let $G$ be an algebraic group over a finite field. Then $G$ sits 
in an extension of algebraic groups
\[ 1 \longrightarrow F \longrightarrow A \times H \longrightarrow G
\longrightarrow 1, \]
where $F$ is finite, $A$ is an abelian variety, and $H$ is affine. 
If $G$ is smooth, then $H$ may be chosen smooth as well.
\end{corollary}

Returning to an arbitrary base field, we finally obtain the existence 
of equivariant compactifications of homogeneous spaces:

\begin{theorem}\label{thm:compactification}
Let $G$ be an algebraic group, and $H$ a closed subgroup. 
Then there exists a projective scheme $X$ equipped with an
action of $G$, and an open $G$-equivariant immersion
$G/H \hookrightarrow X$ with schematically dense image.
\end{theorem}

\begin{proof}
When $G$ is affine, this follows from a theorem of Chevalley 
asserting that $H$ is the stabilizer of a line $L$ in 
a finite-dimensional $G$-module $V$ (see \cite[VIB.11.16]{SGA3}).
Indeed, one may take for $X$ the closure of the $G$-orbit of
$L$ in the projective space of lines of $V$; then $X$ satisfies 
the required properties in view of \cite[III.3.5.2]{DG}. 
Note that $X$ is equipped with an ample $G$-linearized invertible 
sheaf.

When $G$ is proper, the homogeneous space $G/H$ is proper 
as well, and hence is projective by \cite[Cor.~VI.2.6]{Raynaud} 
(alternatively, this follows from the structure of $X$ described
in Remark \ref{rem:proper}).

In the general case, Theorem \ref{thm:chevalley} yields
an affine normal subgroup $N$ of $G$ such that $G/N$ is
proper. Then $N \cdot H$ is a subgroup of $G$, and 
$G/(N \cdot H)$ is proper as well, hence projective. 
It suffices to show the existence of a projective scheme 
$Y$ equipped with an action of $N \cdot H$, an open immersion 
$(N \cdot H)/H \to Y$ with schematically dense image, 
and a $N \cdot H$-linearized ample line bundle: indeed, 
by \cite[Prop.~7.1]{MFK} applied to the projection 
$G \times Y \to Y$ and the $N \cdot H$-torsor 
$G \to G/(N \cdot H)$, the associated fiber bundle 
$G \times^{N \cdot H} Y$ yields the desired equivariant
compactification. In view of Chevalley's theorem used in 
the first step, it suffices in turn to check that
$N \cdot H$ acts on $(N \cdot H)/H$ via an affine quotient
group; equivalently, $(N \cdot H)_{\ant} \subset H$.

By Lemma \ref{lem:quotient}, $(N \cdot H)_{\ant}$
is a quotient of $(N \rtimes H)_{\ant}$. The latter is the fiber
at the neutral element of the affinization morphism
$N \rtimes H \to \Spec \,\cO(N \rtimes H)$. Also, 
$N \rtimes H \cong N \times H$ as schemes, $N$ is affine,
and the affinization morphism commutes with products; thus,
$(N \rtimes H)_{\ant} = H_{\ant}$. As a consequence, 
$(N \cdot H)_{\ant} = H_{\ant}$; this completes the proof.
\end{proof}

\begin{remark}\label{rem:compactification}
When $\char(k) = 0$, the equivariant compactification $X$ of
Theorem \ref{thm:compactification} may be taken smooth, as
follows from the existence of an equivariant desingularization
(see \cite[Prop.~3.9.1, Thm.~3.36]{Kollar}).

In arbitrary characteristics, $X$ may be taken normal if $G$ is 
smooth. Indeed, the $G$-action on any equivariant 
compactification $X$ stabilizes the reduced subscheme $X_{\red}$ 
(since $G \times X_{\red}$ is reduced), and lifts to an action
on its normalization $\tilde{X}$ (since $G \times \tilde{X}$ 
is normal). But the existence of regular compactifications 
(equivariant or not) is an open question.

Over any imperfect field $k$, there exist smooth connected
algebraic groups $G$ having no smooth compactification.
Indeed, we may take for $G$ the subgroup of $\bG_a \times \bG_a$
defined by $y^p -y - t x^p =0$, where $p := \char(k)$ 
and $t \in k \setminus k^p$. This is a smooth affine curve, and hence
has a unique regular compactification $X$.  One checks that $X$ is 
the curve $(y^p - yz^{p-1} - t x^p = 0) \subset \bP^2$, which is not 
smooth at its point at infinity.
\end{remark}

\medskip

\noindent
{\bf Acknowledgements}. I warmly thank Giancarlo Lucchini Arteche
for very helpful e-mail discussions, and Zinovy Reichstein for
his interest and for pointing out several references. Many thanks 
also to the referee for a careful reading of this note, and useful 
comments.

\end{document}